\documentclass[11pt,reqno]{amsart}

 \usepackage{amsmath,amsthm,amsfonts,fullpage}
\usepackage{latexsym,mathabx,txfonts,color}
\usepackage{amssymb}

\newcommand{\R}{\mathbb{R}}

\newcommand{\Z}{\mathbb{Z}}

\renewcommand{\hat}{\widehat}

\newcommand{\norm}[1]{\big\| #1\big\|}

\numberwithin{equation}{section}

\newtheorem{thm}{Theorem}[section]

\newtheorem{lem}[thm]{Lemma}
\newtheorem{prop}[thm]{Proposition}

\theoremstyle{remark}
\newtheorem{rem}{Remark}[section]

\newcommand{\Del}[1]{}

\begin{document}

\title{Small data global regularity for half-wave maps in $n = 4$ dimensions}
\author{Anna Kiesenhofer and Joachim Krieger}



\begin{abstract}
We prove that the half-wave maps problem on $\R^{4+1}$ with target $S^2$ is globally well-posed for smooth initial data which are small in the critical $l^1$ based Besov space. This is a formal analogue of the result \cite{Tat3}. 
\end{abstract}

\maketitle

\section{Introduction}

Denote by $\R^{n+1}$ the Minkowski space of dimension $n+1$, equipped with the standard metric of signature $(1,\,-1,\ldots, -1)$. Let $u: \R^{n+1}\longrightarrow S^2\hookrightarrow \R^3$, with the property that $\nabla_{t,x}u\in L^r(\R^n)$ for some $r\in (1,\infty)$, and furthermore $\lim_{|x|\rightarrow +\infty}u(t, x) = p$ for some fixed $p\in S^2$, for each $t$. Then defining the operator $(-\triangle)^{\frac12}u = -\sum_{j=1}^n (-\triangle)^{-\frac12}\partial_j (\partial_j u)$, we say that $u$ is a {\it{half-wave map}}, provided 
\begin{equation}\label{eq:halfwm}
u_t = u\times (-\triangle)^{\frac12}u.
\end{equation}
 This model was introduced in \cite{KS}, \cite{Le2}, and can be traced to the physics literature, see e. g. \cite{haldane}, as well as the introduction in \cite{KS}. From a mathematical perspective, this model is interesting as it constitutes a close relative of the {\it{Schrodinger maps equation}} on the one hand, given by 
 \[
 u_t = u\times \triangle u,
 \]
 and the classical {\it{wave maps}}, given by 
 \begin{equation}\label{eq:wm}
 \Box u = (-u_t\cdot u_t +\nabla u\cdot \nabla u)u,\,\,\Box = \partial_t^2 - \triangle. 
 \end{equation}
 In fact, the relation to the latter model becomes clearer when re-formulating \eqref{eq:halfwm} as a nonlinear system of wave equations with {\it{non-local source terms}}, as done in detail in section  2 in \cite{KS}: 
 \begin{equation}\label{eq:halfwmwaveform}\begin{split}
 (\partial_t^2 - \triangle) u &= u(\nabla u\cdot\nabla u - \partial_t u\cdot\partial_t u)\\
&+\Pi_{u_{\perp}}\big((-\triangle)^{\frac{1}{2}}u\big)(u\cdot (-\triangle)^{\frac{1}{2}}u)\\
&+u \times (-\triangle)^{\frac{1}{2}}(u \times (-\triangle)^{\frac{1}{2}}u) - u \times(u \times (-\triangle)u)
 \end{split}\end{equation}
 In the preceding equation, the notation $\Pi_{u_{\perp}}$ refers to projection onto the plane orthogonal to $u\in S^2$. Naturally the preceding equation coincides with \eqref{eq:wm}
 except for the last two, non-local terms on the right, and in particular, its scaling behaviour is like that of wave maps. On the other hand, the formally conserved quantity 
 \begin{equation*}
 E(t): = \int_{\R^n}\big|(-\triangle)^{\frac{1}{4}}u\big|^2\,dx
 \end{equation*}
gives an a priori bound only on $\|u\|_{\dot{H}^{\frac12}}$, and so the problem becomes {\it{energy critical}} in dimension $n = 1$. 
\\
In light of the problem \eqref{eq:halfwmwaveform}'s close proximity to wave maps, in particular the fact that it is of the (very) formal form $\Box u = u|\nabla u|^2$, it is then natural to inquire whether fundamental global regularity results for small data 
available for wave maps, such as those obtained in \cite{Tat3}, \cite{T1}, \cite{ShaStr1}, can be similarly established for half-wave maps. This question becomes particularly interesting in low spatial dimensions, where the {\it{null-structure}} in \eqref{eq:wm}, as well as the underlying geometry, play a pivotal role for the theory of wave maps. On the other hand, the 'null-structure' present in \eqref{eq:halfwmwaveform} appears to be of a more complicated nature. In \cite{KS}, the case of large dimensions $n\geq 5$ was handled, in particular establishing sharp space-time bounds for the last two terms in \eqref{eq:halfwmwaveform} using Strichartz estimates as well as the geometry underlying the problem. 
In this paper, we shall push the small data global regularity, relying mostly on Strichartz estimates for free waves, to its limits by settling the case of spatial dimension $n = 4$. Specifically, we settle the issue of global regularity of \eqref{eq:halfwm} for smooth initial data which are small in the {\it{critical Besov space}} $\dot{B}^{2,1}_2$: 
\begin{thm}\label{thm:Main} Let $u(0, \cdot) = u_0: \R^{4}\longrightarrow S^2$ a smooth datum  such that $u_0$ is constant outside of a compact subset of $\R^4$ (this condition in particular ensures that $(-\triangle)^{\frac12}u_0$ is well-defined). Also, assume the smallness condition 
\[
\big\|u_0\big\|_{\dot{B}^{2,1}_2} = \sum_{k\in\Z}2^{2k}\|P_k u\|_{L_x^2(\R^4)}<\epsilon
\]
where $\epsilon\ll 1$ sufficiently small. Then problem \eqref{eq:halfwm} admits a global smooth solution with this datum. 
\end{thm}

\begin{rem} We observe that exploiting the Gauge freedom for this problem as in \cite{T1}, \cite{ShaStr1}, it is likely that one can improve this to global regularity for data small in the critical Sobolev space $\dot{H}^{2}$. Going below spatial dimension $n = 4$ will certainly require new ideas, as the crucial $L_t^2 L_x^\infty$ Strichartz estimate is no longer available, and furthermore it is not a priori clear how the second and third term on the right hand side of \eqref{eq:halfwmwaveform} can be bounded using the commonly used function spaces, involving $X^{s,b}$ as well as null-frame spaces. It does seem, though, that the case $n = 3$ is accessible by the methods of this paper provided one restricts to {\it{radial data}}. 
\end{rem}

The main technical obstacle to overcome in passing from $n = 5$ to $n = 4$ dimensions is the absence of the $L_t^2 L_x^4$-Strichartz estimate, which is useful to handle certain {\it{quadratic interactions}}. This means that in order to control the source terms in \eqref{eq:halfwmwaveform}, absent a strong null-form structure as for the first term on the right, one has to exploit a genuine trilinear structure, which can be revealed upon exploiting the geometric structure, and specifically the fact that $u$ takes values in $S^2$. This is rendered somewhat cumbersome on account of the {\it{non-local character}} of the last two terms in \eqref{eq:halfwmwaveform}, see the proof of Lemma~\ref{lem:complicatedterm} below.  

\section{Technical preliminaries}

Here we recall the key tools used in \cite{KS}.  We let $P_k$, $k\in \Z$, be standard Littlewood-Paley multipliers on $\R^4$ (acting on the spatial variables), and furthermore, we denote by $Q_j$, $j\in \Z$, multipliers which localise a space-time function $F(t, x)$ to dyadic distance $\sim 2^j$ from the light cone $\big|\tau\big| = \big|\xi\big|$ on the Fourier side. Specifically, letting $\tilde{F}(\tau, \xi)$ denote the space time Fourier transform of $F$, while $\hat{f}(\xi)$ denotes the Fourier transform with respect to the spatial variables, and letting $\chi\in C^\infty_0(\R_+)$ a smooth cutoff  satisfying 
\[
\sum_{k\in \Z}\chi \left(\frac{x}{2^k}\right) = 1\,\forall x\in \R_+, 
\]
we set 
\[
\widehat{P_k f}(\xi) = \chi \Big(\frac{|\xi|}{2^k}\Big)\hat{f}(\xi),\,\widetilde{Q_j F} = \chi\Big(\frac{\big||\tau| - |\xi|\big|}{2^j}\Big)\tilde{F}(\tau, \xi). 
\]
Using these ingredients one can then define the following norms: 
\[
\big\|u\big\|_{\dot{X}^{2,\frac12,\infty}}: = \sup_{j\in \Z}2^{\frac{j}{2}}\big\|\nabla_x^{2}Q_j u\big\|_{L_{t,x}^2},\quad \big\|F\big\|_{\dot{X}^{1, -\frac12, 1}}: = \sum_{j\in \Z}2^{-\frac{j}{2}}\big\|\nabla_xQ_j F\big\|_{L_{t,x}^2}
\]
In addition to these, we rely on the classical Strichartz norms, which are the mixed type Lebesgue norms $\big\|\cdot\big\|_{L_t^p L_x^q}$, $\frac1p + \frac{3}{2q}\leq \frac{3}{4}$, $p\geq 2$. Call such pairs $(p,q)$ admissible.
\\
We shall freely use the fact that Fourier localisers of the form $P_k Q_j$ act in bounded fashion on spaces of the form $L_t^p L_x^2$, $1\leq p\leq \infty$, see e. g. \cite{T2}. 
We can now define a norm controlling our solutions as follows: 
\begin{equation*}\label{eq:Snorm}
\big\|u\big\|_{S}: = \sum_{k\in \Z}\sup_{(p,q)\,\text{admissible}}2^{(\frac1p + \frac{4}{q}-1)k}\big\|\nabla_{t,x}P_ku\big\|_{L_t^p L_x^q} + \big\|\nabla_{t,x}P_k u\big\|_{\dot{X}^{1,\frac12,\infty}} =:\sum_{k\in \Z}\big\|P_k u\big\|_{S_k}.
\end{equation*}
We also introduce 
\begin{equation*}\label{eq:Nnorm}
\big\|F\big\|_{N}: = \sum_{k\in \Z}\big\|P_kF\big\|_{L_t^1\dot{H}^{1} + \dot{X}^{1,-\frac12,1}},
\end{equation*}
as well as the norms 
\[
\big\|u\big\|_{\dot{B}^{r,1}_2}: = \sum_{k\in \Z}\big\|P_ku\big\|_{\dot{H}^r}
\]

Then the following inequality is by now completely standard, see e. g. \cite{Kri}, \cite{T1}, \cite{Tat3}: 
\begin{prop}\label{prop:linestimates}
\begin{equation*}\label{eq:energyinequality}
\big\| u\big\|_S\lesssim \big\|u[0]\big\|_{\dot{B}^{2,1}_2\times \dot{B}^{1,1}_2} + \big\|\Box u\big\|_{N}. 
\end{equation*}
\end{prop}

We shall also frequently rely on {\it{Bernstein's inequality}}. In our context, we need that if $1\leq p\leq q\leq \infty$, we have 
\[
\| P_k f\|_{L^q(\R^n)}\lesssim 2^{(\frac{n}{p} - \frac{n}{q})k}\|P_k f\|_{L^p(\R^n)}.
\]

\section{The main estimates}

The following estimate will be the key towards finding solutions of the wave equation \eqref{eq:halfwmwaveform} by means of iteration. The key here is particularly the estimate \eqref{eq:multilin3}, whose proof requires a refinement of the analogous Prop. 4.1 in \cite{KS}. 

\begin{prop}\label{prop:multilinear} Assume that $u$ takes values in $S^2$ and converges to $p\in S^2$ at spatial infinity. Then using the norms $\big\|\cdot\big\|_{S}, \big\|\cdot\big\|_{N}$ introduced in the previous section, we have for suitable $\sigma>0$ the bounds 
\begin{equation}\label{eq:multilin1}
\big\|P_k\big[u(\nabla u\cdot\nabla u - \partial_t u\cdot\partial_t u)\big]\big\|_{N}\lesssim (1+\big\|u\big\|_{S})\big\|u\big\|_{S}\big(\sum_{k_1\in \Z}2^{-\sigma|k-k_1|}\big\|P_{k_1}u\big\|_{S_{k_1}}\big)
\end{equation}
Furthermore, if $\tilde{u}$ maps into a small neighbourhood of $S^2$ and $\big\|\tilde{u}\big\|_{S}\lesssim 1$, we have the similar bound 
\begin{equation}\label{eq:multilin2}
\big\|P_k\big(\Pi_{\tilde{u}_{\perp}}\big((-\triangle)^{\frac{1}{2}}u\big)(u\cdot (-\triangle)^{\frac{1}{2}}u)\big)\big\|_{N}\lesssim \prod_{v = u,\tilde{u}} (1+\big\|v\big\|_{S})\big\|u\big\|_{S}\big(\sum_{k_1\in \Z}2^{-\sigma|k-k_1|}\big\|P_{k_1}u\big\|_{S_{k_1}}\big)
\end{equation}
as well as 
\begin{equation}\label{eq:multilin3}\begin{split}
&\big\|P_k\big(\Pi_{\tilde{u}_{\perp}}\big[u \times (-\triangle)^{\frac{1}{2}}(u \times (-\triangle)^{\frac{1}{2}}u) - u \times(u \times (-\triangle)u)\big]\big)\big\|_{N}\\
&\lesssim \prod_{v = u,\tilde{u}} (1+\big\|v\big\|_{S})\big\|u\big\|_{S}\big(\sum_{k_1\in \Z}2^{-\sigma|k-k_1|}\big\|P_{k_1}u\big\|_{S_{k_1}}\big).
\end{split}\end{equation}
We also have corresponding difference estimates: assuming that $u^{(j)}, j = 1,2$, map into $S^2$, while $\tilde{u}^{(j)}$ map into a small neighbourhood of $S^2$ and with $\lim_{N\rightarrow +\infty}P_{<-N}\tilde{u}^{(j)} = p\in S^2$, then using the notation 
\[
\triangle_{1,2}F^{(j)} = F^{(1)} - F^{(2)}
\]
we have 
\begin{equation}\label{eq:multilin4}\begin{split}
&\big\|\triangle_{1,2}P_k\big[u^{(j)}(\nabla u^{(j)}\cdot\nabla u^{(j)} - \partial_t u^{(j)}\cdot\partial_t u^{(j)})\big]\big\|_{N}\\&\lesssim (1+\max_j\big\|u^{(j)}\big\|_{S})(\max_j\big\|u^{(j)}\big\|_{S})\big(\sum_{k_1\in \Z}2^{-\sigma|k-k_1|}\big\|P_{k_1}u^{(1)} - P_ku^{(2)}\big\|_{S_{k_1}}\big)\\
& + (1+\max_j\big\|u^{(j)}\big\|_{S})(\big\|u^{(1)} - u^{(2)}\big\|_{S})\big(\max_j\sum_{k_1\in \Z}2^{-\sigma|k-k_1|}\big\|P_{k_1}u^{(j)}\big\|_{S_{k_1}}\big)\\
\end{split}\end{equation}
and similarly 
\begin{equation}\label{eq:multilin5}\begin{split}
&\big\|P_k\triangle_{1,2}\big(\Pi_{\tilde{u}^{(j)}_{\perp}}\big((-\triangle)^{\frac{1}{2}}u^{(j)}\big)(u^{(j)}\cdot (-\triangle)^{\frac{1}{2}}u^{(j)})\big)\big\|_{N}\\&\lesssim \max_j\prod_{v = u^{(j)},\tilde{u}^{(j)}} (1+\big\|v\big\|_{S})\big\|u^{(j)}\big\|_{S}\big(\sum_{k_1\in \Z}2^{-\sigma|k-k_1|}\big\|P_{k_1}u^{(1)} - P_{k_2}u^{(2)} \big\|_{S_{k_1}}\big)\\
& + \max_j\prod_{v = u^{(j)},\tilde{u}^{(j)}} (1+\big\|v\big\|_{S})\big\|u^{(1)}- u^{(2)}\big\|_{S}\big(\max_j\sum_{k_1\in \Z}2^{-\sigma|k-k_1|}\big\|P_{k_1}u^{(j)}\big\|_{S_{k_1}}\big)\\
& +  \max_j(1+\big\|u^{(j)}\big\|_{S})\big\|\tilde{u}^{(1)}- \tilde{u}^{(2)}\big\|_{S}\big(\max_j\sum_{k_1\in \Z}2^{-\sigma|k-k_1|}\big\|P_{k_1}u^{(j)}\big\|_{S_{k_1}}\big)\\
\end{split}\end{equation}
The analogous difference estimate for \eqref{eq:multilin3} is similar. 

\end{prop}

The proof of the estimates for the wave maps term, \eqref{eq:multilin1},  \eqref{eq:multilin4},  is  completely standard, see e. g. \cite{Tat3}, and the proof of Prop. 4.1 in \cite{KS}. Therefore we shall only show the other estimates  here. Note that for these estimates we do not need the $X^{s,b}$ part of the norms $\norm{\cdot}_S$ and $\norm{\cdot}_N$ but can get by using only Strichartz estimates. In particular, we do not need the frequency localizers $Q_j$. For the frequency localizers $P_k$ we also use the notation $u_k:=P_k u$.

A technical tool in the proof of this proposition is the following lemma, analogous to Lemma 3.1 used in \cite{KS}: 

\begin{lem}\label{lem:KS}Assume that $\|\cdot\|_{X, Y, Z}$ are translation invariant norms on $C^{\infty}_c(\R^n)$ satisfying an inequality 
\[
\big\| uv\big\|_{Z}\leq \big\|u\big\|_{X}\big\|v\big\|_{Y}. 
\]
Let 
$$(F_{k_1 k_2}(u,v))(x)= \int_{\R^n \times \R^n} m(\xi, \eta) \hat u_{k_1} \hat v_{k_2} e^{ix \cdot (\xi + \eta)}d(\xi,\eta)$$
where $k_1,k_2 \in \Z$ and the multiplier $m(\xi, \eta)$ is $C^{\infty}$ on the support of $\chi_{k_1}(\xi) \chi_{k_2}(\eta)$ and on this set satisfies 
$$
 | (2^{k_1}\nabla_\xi)^i (2^{k_2}\nabla_\eta)^j m(\xi,\eta)| \lesssim_{i,j} \gamma_{k_1 k_2}\, \forall i,j\ge 0.$$
Then 
$$\norm{F_{k_1 k_2}(u,v)}_{Z} \lesssim \gamma_{k_1 k_2} \big\|u_{k_1}\big\|_{X}\big\|v_{k_2}\big\|_{Y}$$
where the implied constant depends on the size of finitely many derivatives of $m$.

\end{lem}

\begin{proof} As in \cite{KS}, write the multiplier $m(\xi,\eta)$ as a Fourier series on the support of $\chi_{k_1}(\xi) \chi_{k_2}(\eta)$,
$$m(\xi,\eta)=\sum_{m,p\in\Z^n} a_{mp} e^{i (2^{-k_1}\xi\cdot m + 2^{-k_2}\eta \cdot p)}.$$
Then $(F_{k_1 k_2}(u,v))(x)=\sum_{m,p\in\Z^n} a_{mp} u_{k_1}(x+2^{-k_1}m) v_{k_2}(x+2^{-k_2} p)$. By taking $N$th order derivatives of $m$ we see that the Fourier coefficients $a_{mp}$ satisfy
$ 
 |a_{mp}| \lesssim_N \gamma_{k_1 k_2} (|m|+|p|)^{-N}$.
Hence $\sum_{m,p\in \Z^n} |a_{mp}| \lesssim \gamma_{k_1 k_2}$ and the estimate for $F_{k_1 k_2}$ follows. 
\end{proof}

Applying Lemma \ref{lem:KS} to the multipliers
$$m(\xi,\eta)=\chi_0(\xi +\eta)(|\xi +\eta|-|\eta|) \quad\text{resp.} \quad m(\xi,\eta)=\chi_0(\xi +\eta)|\eta|(|\xi +\eta|-|\eta|)$$ 
we obtain for $k_1,k_2 \lesssim 1$:
\begin{equation}\label{singlefreq}
\big\|P_0 \left( (-\triangle)^{\frac12}(u_{k_1}\cdot u_{k_2} ) - u_{k_1}\cdot  (-\triangle^{\frac12}u_{k_2})\right) \big\|_{Z}\lesssim2^{k_1} \big\| u_{k_1}\big\|_{X}\big\|u_{k_2}\big\|_{Y}
\end{equation}
\begin{equation}\label{singlefreq2}
\big\|P_0\left((-\triangle)^{\frac12}(u_{k_1}\cdot (-\triangle)^{\frac12}u)-u_{k_1}\cdot (\triangle u_{k_2}) \right)\big\|_{Z}\lesssim 2^{k_1+k_2} \big\| u_{k_1}\big\|_{X}\big\|u_{k_2}\big\|_{Y}.
\end{equation}


To establish the bounds \eqref{eq:multilin2}, \eqref{eq:multilin3}, we shall rely on a simple paradifferential identity which derives from the crucial feature that 
\[
u\cdot u = 1,
\]
which gets localised to 
\[
P_0(u\cdot u) = 0.
\]
We perform a frequency decomposition: 
\begin{align*}
P_0(u\cdot u) &= P_0(P_{<-10}u\cdot u) + P_0(P_{\geq -10}u\cdot u)\\
& = P_0(P_{<-10}u\cdot P_{\geq -10} u) + P_0(P_{\geq -10}u\cdot P_{<-10}u)\\& + P_0(P_{\geq -10}u\cdot P_{\geq -10}u)\\
& = 2P_0(P_{<-10}u\cdot P_{\geq -10} u) + P_0(P_{\geq -10}u\cdot P_{\geq -10}u).
\end{align*}
It follows that we have the identity 
\begin{equation}\label{eq:key}
P_0(P_{<-10}u\cdot P_{\geq -10} u) = -\frac12 P_0(P_{\geq -10}u\cdot P_{\geq -10}u).
\end{equation}
More generally the same reasoning leads to
\begin{equation}\label{eq:key_general}
P_{m} (P_{<k}u\cdot P_{\geq k} u) = -\frac12 P_{m} (P_{\geq k}u\cdot P_{\geq k}u).
\end{equation}
for any $m \ge k+3$.
\\

This can be used as follows. Summing \eqref{singlefreq} over $k_1<-10,k_2\in [-2,2]$ and using Equation \eqref{eq:key} we obtain 
\begin{align}\label{eq:key2}\begin{split}
\big\|P_0&(u_{<-10}\cdot (-\triangle)^{\frac12}u_{})\big\|_{Z} = \big\|P_0(u_{<-10}\cdot (-\triangle)^{\frac12}u_{\geq -10})\big\|_{Z} \\ &\le
\norm{P_0(-\triangle)^{\frac12}(u_{< -10}u\cdot u_{\geq -10})}_Z +
 \big\|P_0((-\triangle)^{\frac12}(u_{< -10}\cdot u_{\geq -10}) -  u_{<-10}\cdot (-\triangle)^{\frac12}u_{\geq -10})\big\|_{Z}
\\&  \lesssim\norm{P_0(-\triangle)^{\frac12}(u_{\geq -10}\cdot u_{\geq -10})}_Z+  \sum_{\substack{k_1<-10,\\ k_2\in[-2,2]}} 2^{k_1}\big\|  u_{k_1}\big\|_{X}\big\|u_{k_2}\big\|_{Y}.
\end{split}
\end{align}
Similarly, combining \eqref{singlefreq} and \eqref{singlefreq2} we obtain
\begin{align}\label{eq:key3}\begin{split}
\big\|P_0&(u_{<-10}\cdot (-\triangle u_{}))\big\|_{Z}   \lesssim\norm{P_0(-\triangle)(u_{\geq -10}\cdot u_{\geq -10})}_Z+  \sum_{\substack{k_1<-10,\\ k_2\in[-2,2]}} 2^{k_1+k_2}\big\|  u_{k_1}\big\|_{X}\big\|u_{k_2}\big\|_{Y}.
\end{split}
\end{align}
Rescaling Equations \eqref{eq:key2} and \eqref{eq:key3} for $Z=L_x^p$, $X=L_x^{p_1}$, $Y= L_x^{p_2}$ (where $\frac{1}{p}= \frac{1}{p_1}+\frac{1}{p_2}$) we get:
\begin{align}\label{rescaled}
\big\|P_k\big(u_{<k-10}\cdot (-\triangle)^{\frac12}u\big)\big\|_{L_x^p}
&\lesssim \big\|P_k (-\triangle)^{\frac12}\big(u_{\geq k-10}\cdot u_{\geq k-10}\big)\big\|_{L_x^{p}} +   \sum_{\substack{k_1<k-10\\ k_2\in [k-2,k+2]}} 2^{k_1} \big\| u_{k_1} \big\|_{L_x^{p_1}}\big\|u_{k_2}\big\|_{L_x^{p_2}} 
\\ \label{rescaled2}
\big\|P_k(u_{<k-10}\cdot (-\triangle u_{}))\big\|_{L_x^p}   &\lesssim\norm{P_k(-\triangle)(u_{\geq k-10}\cdot u_{\geq k-10})}_{L_x^p}+  \sum_{\substack{k_1<k-10,\\ k_2\in[k-2,k+2]}} 2^{k+k_1}\big\|  u_{k_1}\big\|_{L_x^{p_1}}\big\|u_{k_2}\big\|_{L_x^{p_2}}.
\end{align}
for all $k\in \Z$. 
\\

The proof of the preceding proposition will be accomplished in a sequence of lemmas: 

\begin{lem}\label{lem:multilin2} Let $u, \tilde{u}$ be as in the statement of the proposition. Then we have the bound 
\[
\big\|P_0\big(\Pi_{\tilde{u}^{\perp}}\big((-\triangle)^{\frac12}u\big)\big(u\cdot (-\triangle)^{\frac12}u\big)\big)\big\|_{L_t^1 \dot{H}_x^1}\lesssim \sum_{k\in Z}2^{-\sigma |k|}\|P_ku\|_{S}(1+\|\tilde{u}\|_{S})\big\|u\big\|_{S}^2
\]
for suitable $\sigma>0$. 
In particular, we have the bound 
\[
\big\|\Pi_{\tilde{u}^{\perp}}\big((-\triangle)^{\frac12}u\big)\big(u\cdot (-\triangle)^{\frac12}u\big)\big\|_{L_t^1 \dot{B}^{1,1}_2}\lesssim (1+\|\tilde{u}\|_{S})\big\|u\big\|_{S}^3. 
\]
\end{lem}
\begin{proof} In the sequel, the norm $\|\cdot\|_{S}$ will only refer to the Strichartz type mixed Lebesgue norms in the original definition of $\|\cdot\|_{S}$, and with $\nabla_{t,x}$ replaced by $\nabla_x$. This weaker norm will suffice to bound things. To begin with, we observe that 
\begin{equation}\label{eq:piu}
\big\|P_0\big(\Pi_{\tilde{u}^{\perp}}\big((-\triangle)^{\frac12}u\big)\big\|_{S}\lesssim (1+\big\|\tilde{u}\big\|_{S})\sum_{k_1\in Z}2^{-\sigma_1|k_1|}\big\|P_{k_1}u\big\|_{S}
\end{equation}
for suitable $\sigma_1>0$, and by re-scaling we have 
\begin{equation}\label{eq:piu1}
\big\|P_k\big(\Pi_{\tilde{u}^{\perp}}\big((-\triangle)^{\frac12}u\big)\big\|_{S}\lesssim 2^k(1+\big\|\tilde{u}\big\|_{S})\sum_{k_1\in Z}2^{-\sigma_1|k_1-k|}\big\|P_{k_1}u\big\|_{S}
\end{equation}
To see this bound, write 
\begin{align*}
P_0\big(\Pi_{\tilde{u}^{\perp}}\big((-\triangle)^{\frac12}u\big) = P_0(-\triangle)^{\frac12}u - P_0\big((\frac{\tilde{u}}{\|\tilde{u}\|}\cdot(-\triangle)^{\frac12}u)\frac{\tilde{u}}{\|\tilde{u}\|}\big)
\end{align*}
Then further decompose 
\begin{align*}
P_0\big((\frac{\tilde{u}}{\|\tilde{u}\|}\cdot(-\triangle)^{\frac12}u)\frac{\tilde{u}}{\|\tilde{u}\|}\big) &= P_0\big((\frac{\tilde{u}}{\|\tilde{u}\|}\cdot(-\triangle)^{\frac12}u_{>10})\frac{\tilde{u}}{\|\tilde{u}\|}\big)\\
& + P_0\big((\frac{\tilde{u}}{\|\tilde{u}\|}\cdot(-\triangle)^{\frac12}u_{[-10,10]})\frac{\tilde{u}}{\|\tilde{u}\|}\big)\\
&+ P_0\big((\frac{\tilde{u}}{\|\tilde{u}\|}\cdot(-\triangle)^{\frac12}u_{<-10})\frac{\tilde{u}}{\|\tilde{u}\|}\big)\\
\end{align*}
Then use for the first term on the right that 
\begin{align*}
&\big\|P_0\big((\frac{\tilde{u}}{\|\tilde{u}\|}\cdot(-\triangle)^{\frac12}u_{>10})\frac{\tilde{u}}{\|\tilde{u}\|}\big)\big\|_{L_t^2L_x^\infty\cap L_t^\infty L_x^2}\\&\leq \sum_{k_1>10}\big\|P_0\big(([\frac{\tilde{u}}{\|\tilde{u}\|}]_{>k_1-10}\cdot(-\triangle)^{\frac12}u_{k_1})\frac{\tilde{u}}{\|\tilde{u}\|}\big)\big\|_{L_t^2L_x^\infty\cap L_t^\infty L_x^2}\\
&+\sum_{k_1>10}\big\|P_0\big(([\frac{\tilde{u}}{\|\tilde{u}\|}]_{\leq k_1-10}\cdot(-\triangle)^{\frac12}u_{k_1})[\frac{\tilde{u}}{\|\tilde{u}\|}]_{>k_1-5}\big)\big\|_{L_t^2L_x^\infty\cap L_t^\infty L_x^2}\\
&\lesssim  \sum_{k_1>10}\sum_{a=5,10}\big\|[\frac{\tilde{u}}{\|\tilde{u}\|}]_{>k_1-a}\big\|_{L_t^\infty L_x^2}\big\|(-\triangle)^{\frac12}u_{k_1}\big\|_{L_t^2L_x^\infty\cap L_t^\infty L_x^2}\\
&\lesssim \sum_{k_1>10}2^{-\frac{k_1}{2}}\big\|u_{k_1}\big\|_{S},
\end{align*}
where we have used Holder's and Bernstein's inequality. 
For the second term on the right, we have 
\begin{align*}
\big\|P_0\big((\frac{\tilde{u}}{\|\tilde{u}\|}\cdot(-\triangle)^{\frac12}u_{[-10,10]})\frac{\tilde{u}}{\|\tilde{u}\|}\big)\big\|_{L_t^2L_x^\infty\cap L_t^\infty L_x^2}&\lesssim \big\|(-\triangle)^{\frac12}u_{[-10,10]}\big\|_{L_t^2L_x^\infty\cap L_t^\infty L_x^2}\\
&\lesssim \sum_{k_1\in [-10,10]}2^{-\frac{|k_1|}{2}}\big\|u_{k_1}\big\|_{S}.
\end{align*}
Finally, for the last term on the right, we get 
\begin{align*}
&\big\|P_0\big((\frac{\tilde{u}}{\|\tilde{u}\|}\cdot(-\triangle)^{\frac12}u_{<-10})\frac{\tilde{u}}{\|\tilde{u}\|}\big)\big\|_{L_t^2L_x^\infty\cap L_t^\infty L_x^2}\\
&\lesssim \big\|(-\triangle)^{\frac12}u_{<-10}\big\|_{L_t^2 L_x^\infty\cap L_{t,x}^{\infty}}\big\|[\frac{\tilde{u}}{\|\tilde{u}\|}]_{>-10}\big\|_{L_t^\infty L_x^2}\\
&\lesssim  \sum_{k_1<-10}2^{-\frac{|k_1|}{2}}\big\|u_{k_1}\big\|_{S}.
\end{align*}
The bound \eqref{eq:piu} then follows in light of the fact that all the norms in $\|\cdot\|_{S}$ are obtained by interpolating between $L_t^2 L_x^\infty$ and $L_t^\infty L_x^2$ and using Bernstein's inequality. 
\\
Proceeding with the proof of the lemma, it suffices to prove a bound of the form 
\begin{equation}\label{eq:lemma1}
\big\|P_0\big((P_{k_1}\Pi_{\tilde{u}^{\perp}})\big((-\triangle)^{\frac12}u\big) \big(u\cdot (-\triangle)^{\frac12}u\big)\big\|_{L_t^1 \dot{H}_x^1}\lesssim 2^{-\sigma |k_1|}\sum_k 2^{-\sigma|k-k_1|}\big\|u_{k}\big\|_{S}\big\|u\big\|_{S}^2, 
\end{equation}
for suitable $\sigma>0$, since we can then bound the right hand side by
$$2^{-\frac{\sigma}{2}|k_1|} \sum_k 2^{-\frac{\sigma}{2}|k|}\big\|u_{k}\big\|_{S}\big\|u\big\|_{S}^2$$
and summing over $k_1$ we obtain the statement of the lemma.

We show \eqref{eq:lemma1} by distinguishing between a number of cases: 
\\

{\it{(i): $k_1<-10$.}} We place $(P_{k_1}\Pi_{\tilde{u}^{\perp}})\big((-\triangle)^{\frac12}u\big)$ into $L_t^2 L_x^\infty$, 
\begin{align*}
&\big\|P_0\big((P_{k_1}\Pi_{\tilde{u}^{\perp}})\big((-\triangle)^{\frac12}u\big)\big(u\cdot (-\triangle)^{\frac12}u\big)\big\|_{L_t^1 \dot{H}_x^1}\\& \lesssim \norm{(P_{k_1}\Pi_{\tilde{u}^{\perp}})\big((-\triangle)^{\frac12}u\big)}_{L_t^2 L_x^\infty} \big\|P_0\big(u_{}\cdot (-\triangle)^{\frac12}u\big)\big\|_{L_t^2L_x^{2}}\\& \lesssim 2^{\frac{k_1}{2}} \big(\sum_{k_2}2^{-\sigma_1|k_1-k_2|}\|u_{k_2}\|_{S}\big) \big\|P_0\big(u_{}\cdot (-\triangle)^{\frac12}u\big)\big\|_{L_t^2L_x^{2}},
\end{align*}
where we have used \eqref{eq:piu1} and the definition of $S$ for the last inequality. 
 
 We are thus left with showing that
$\big\|P_0\big(u_{}\cdot (-\triangle)^{\frac12}u\big)\big\|_{L_t^2L_x^{2}}\lesssim\big\|u\big\|_{S}^2. $ 
If the first factor $u$ has large frequency the estimate is straightforward:
\begin{align*}
\big\|P_0\big(u_{\geq -10}\cdot (-\triangle)^{\frac12}u\big)\big\|_{L_t^2L_x^{2}}&\le \sum_{|k_1-k_2|<10,\,k_1>10}\norm{u_{k_1}}_{L_t^\infty L_x^2} \norm{(-\triangle)^{\frac12}u_{k_2}}_{L_t^2 L_x^\infty}\\
&+\sum_{k_1\in [-10,10],\,k_2<20}\norm{u_{k_1}}_{L_t^\infty L_x^2} \norm{(-\triangle)^{\frac12}u_{k_2}}_{L_t^2 L_x^\infty}\\& \lesssim \big\|u\big\|_{S}^2. 
\end{align*}
If $u$ has low frequency we use Equation \eqref{eq:key2} to turn it into high frequency:
\begin{align*}
\big\|P_0\big(P_{<-10}u\cdot (-\triangle)^{\frac12}u\big)\big\|_{L_t^1 \dot{H}_x^1}\lesssim 
\big\|P_0(-\triangle^{\frac12}(P_{\geq -10}u\cdot P_{\geq -10}u)\big\|_{L_t^2 L_x^2}
+ \sum_{\substack{\ell <-10,\\ k_2\in[-2,2]}} 2^{\ell} \big\|u_\ell \big\|_{L_t^2L_x^\infty}\big\|u_{k_2}\big\|_{L_t^\infty L_x^2}.
\end{align*}
This expression now is easily seen to be bounded by $\norm{u}_S^2$ and so the case $k_1 <-10$ is finished. 
\\

{\it{(ii): $k_1\ge -10$.}} Now we place $(P_{k_1}\Pi_{\tilde{u}^{\perp}})\big((-\triangle)^{\frac12}u\big)$ into $L_t^\infty L_x^2$
\begin{align*}
&\big\|P_0\big((P_{k_1}\Pi_{\tilde{u}^{\perp}})\big((-\triangle)^{\frac12}u\big) \big(u\cdot (-\triangle)^{\frac12}u\big)\big\|_{L_t^1 \dot{H}_x^1}\\&
 \le \norm{(P_{k_1}\Pi_{\tilde{u}^{\perp}})\big((-\triangle)^{\frac12}u\big)}_{L_t^{\infty}L_x^2}  \norm{u\cdot (-\triangle)^{\frac12}u}_{L_t^1 L_x^{\infty}}\\
 & \le 2^{-k_1}\sum_{k}2^{-\sigma_1|k-k_1|}\big\|u_{k}\big\|_S \norm{u\cdot (-\triangle)^{\frac12}u}_{L_t^1 L_x^{\infty}}.
\end{align*}
We are left with showing that $\norm{u\cdot (-\triangle)^{\frac12}u}_{L_t^1 L_x^{\infty}}\lesssim \norm{u}_S^2.$ Decompose 
\begin{align}\label{case2}
\begin{split}
u\cdot (-\triangle)^{\frac12}u
= \sum_{k\in \Z} P_k \big(u_{<k-10}\cdot (-\triangle)^{\frac12}u + u_{\geq k-10}\cdot (-\triangle)^{\frac12}u_{<k-10} + u_{\geq k-10}\cdot (-\triangle)^{\frac12}u_{\geq k-10}\big) 
\end{split}
\end{align}
Then estimate the first term by using Equation \eqref{rescaled}:
\begin{align*}
\big\|P_k\big(u_{<k-10}\cdot (-\triangle)^{\frac12}u\big)\big\|_{L_t^1 L_x^\infty}
\lesssim \big\|P_k (-\triangle)^{\frac12}\big(u_{\geq k-10}\cdot u_{\geq k-10}\big)\big\|_{L_t^1 L_x^\infty}
 +  \sum_{\substack{\ell<k-10\\ k_2\in [k-2,k+2]}} 2^{\ell} \big\|u_{\ell} \big\|_{L_t^2 L_x^\infty}\big\|u_{k_2}\big\|_{L_t^2 L_x^\infty}
\end{align*}
Summing over $k\in \Z$ we get for the first term on the right
\begin{align*}
\sum_{k\in \Z}\big\|&P_k (-\triangle)^{\frac12}\big(u_{\geq k-10}\cdot u_{\geq k-10}\big)\big\|_{L_t^1 L_x^\infty} \lesssim 
\sum_{k\in \Z} 2^{\frac{k}{2}} \norm{u_{\ge k-10}}_{L_t^2 L_x^\infty}\norm{u}_{S} \le \\
& \sum_{k\in \Z} \sum_{\ell \ge k-10} 2^{\frac{k-\ell}{2}} \norm{u_\ell}_{S} \norm{u}_{S}
 \lesssim \norm{u}_{S}^2;
\end{align*}
for the second term the same bound is seen to hold so 
\begin{align*}
\sum_{k\in \Z}\big\|P_k\big(u_{<k-10}\cdot (-\triangle)^{\frac12}u\big)\big\|_{L_t^1 L_x^\infty}\lesssim \big\|u\big\|_{S}^2. 
\end{align*}

Now for the second term in brackets in Equation \eqref{case2}
 we directly see that 
\begin{align*}
\sum_{k\in \Z} \norm{P_{k} \big(u_{\geq k-10}\cdot (-\triangle)^{\frac12}u_{< k-10}\big)}_{L^1_t L^\infty_x}\le  \sum_{k\in \Z} \norm{u_{[k-5, k+5]}}_{L^2_t L^\infty_x} \norm{ (-\triangle)^{\frac12}u_{< k-10}}_{L^2_t L^\infty_x} \lesssim \big\|u\big\|_{S}^2. 
\end{align*}
It remains to estimate the last term in brackets in Equation \eqref{case2}:
 $$\norm{\sum_{\substack{\ell,m
\in \Z \\|\ell-m|\le 20}} u_{\ell}\cdot (-\triangle)^{\frac12}u_{m} }_{L_t^1 L_x^\infty}\le 
\sum_{\substack{\ell,m
\in \Z \\|\ell-m|\le 20}} \norm{u_{\ell}}_{L_t^2 L_x^\infty} \norm{(-\triangle)^{\frac12}u_{m}}_{L_t^2 L_x^\infty}
\lesssim \sum_{m
\in \Z } \norm{u}_S \norm{u_m}_S 
= \norm{u}_S^2.$$

We conclude that 
\begin{align*}
&\big\|P_0\big((P_{k_1}\Pi_{\tilde{u}^{\perp}})\big((-\triangle)^{\frac12}u\big) \big(u\cdot (-\triangle)^{\frac12}u\big)\big\|_{L_t^1 \dot{H}_x^1}
\lesssim 2^{-\frac{k_1}{2}}\sum_k2^{-\sigma |k|}\big\|u_{k}\big\|_S\big\|u\big\|_{S}^2.
\end{align*}


\end{proof}

We next prove \eqref{eq:multilin3}. This will follow from the next lemma, after a simple re-scaling:
\begin{lem}\label{lem:complicatedterm} Under the assumptions of the preceding proposition, we have the bound 
\begin{align*}
&\big\|(P_0\Pi_{\tilde{u}^{\perp}})\big[u\times (-\triangle)^{\frac12}(u \times (-\triangle)^{\frac12}u) - u\times (u\times (-\triangle u))\big]\big\|_{L_t^1 \dot{H}_x^1}\\&
\lesssim \sum_{k\in \Z} 2^{-\sigma |k|} \big\|u_{k}\big\|_{S}\big\|u\big\|_{S}^2 
\end{align*}
for suitable $\sigma>0$. 
\end{lem}
\begin{proof} This is more delicate than the previous proof, since we have to take advantage of the relation \eqref{eq:key}, but the operator $(-\triangle)^{\frac12}$ acts non-locally, causing the factors in \eqref{eq:key} to be evaluated at different points. 

Observe that it is enough to prove the bound
\begin{align}\label{eq:boundthis}\begin{split}
&\big\|P_0\big[u\times (-\triangle)^{\frac12}(u_{k_1}\times (-\triangle)^{\frac12}u) - u\times (u_{k_1}\times (-\triangle u))\big]\big\|_{L_t^1 \dot{H}_x^1}\\
&\lesssim 2^{-\sigma |k_1|} \big\|u_{k_1}\big\|_{S}\big\|u\big\|_{S}^2 + \norm{u_{k_1}}_S \norm{u_{0}}_S\norm{u}_S.
\end{split}
\end{align}
To get rid of the projection operator $\Pi_{\tilde{u}^{\perp}}$ we use the inequality 
\begin{align*}
\big\|(P_0\Pi_{\tilde{u}^{\perp}})P_kF\big\|_{L_t^1 \dot{H}_x^1}\lesssim 2^{-\sigma_2|k|}\big\|P_kF\big\|_{L_t^1 \dot{H}_x^1}, 
\end{align*}
for suitable $\sigma_2>0$, which is proved exactly like \eqref{eq:piu}. 
Assuming then Equation \eqref{eq:boundthis} and re-scaling, we  get the bound 
\begin{align*}
&\big\|(P_0\Pi_{\tilde{u}^{\perp}}) \big(u\times (-\triangle)^{\frac12}(u_{k_1}\times (-\triangle)^{\frac12}u) - u\times (u_{k_1}\times (-\triangle u)\big) \big\|_{L_t^1 \dot{H}_x^1}\\&
\lesssim \sum_{k\in \Z} 2^{-\sigma_2|k|} \big(2^{-\sigma |k_1|} \norm{u_{k_1+k}}_S \norm{u}^2_S + \norm{u_{k_1+k}}_S \norm{u_k}_S \norm{u}_S \big)
\end{align*}
which yields the assertion of the lemma after summing over $k_1$.


To prove \eqref{eq:boundthis} we again split into different cases depending on $k_1$: 
\\

{\it{(i): $k_1<-10$.}} Decompose into an easy and a more delicate term: 
\begin{align*}
&P_0\big[u\times (-\triangle)^{\frac12}(u_{k_1}\times (-\triangle)^{\frac12}u) - u\times (u_{k_1}\times (-\triangle u))\big]\\
&= P_0\big[u_{\geq k_1-10}\times (-\triangle)^{\frac12}(u_{k_1}\times (-\triangle)^{\frac12}u) - u_{\geq k_1-10}\times (u_{k_1}\times (-\triangle u))\big]\\
& + P_0\big[u_{<k_1-10}\times (-\triangle)^{\frac12}(u_{k_1}\times (-\triangle)^{\frac12}u) - u_{<k_1-10}\times (u_{k_1}\times (-\triangle u))\big]\\
\end{align*}
Then we bound the first term on the right as follows: we have 
\begin{align*}
&P_0\big[u_{[k_1-10, -10]}\times (-\triangle)^{\frac12}(u_{k_1}\times (-\triangle)^{\frac12}u) - u_{[k_1-10, -10]}\times (u_{k_1}\times (-\triangle u))\big]\\
& = P_0\big[u_{[k_1-10, -10]}\times (-\triangle)^{\frac12}(u_{k_1}\times (-\triangle)^{\frac12}u_{[-2,2]}) - u_{[k_1-10, -10]}\times (u_{k_1}\times (-\triangle u_{[-2,2]}))\big]\\
\end{align*}
Applying the multiplier lemma, i.e. Equation \eqref{singlefreq2}, we find
\begin{align*}
&\big\|(-\triangle)^{\frac12}(u_{k_1}\times (-\triangle)^{\frac12}u_{[-2,2]}) - (u_{k_1}\times (-\triangle u_{[-2,2]}))\big\|_{L_{t,x}^2}\\
&\lesssim 2^{{k_1}}\big\|u_{k_1}\big\|_{L_t^2 L_x^\infty}\big\|u_{[-2,2]}\big\|_{L_t^\infty L_x^2}\le  2^{\frac{k_1}{2}}\big\|u_{k_1}\big\|_{S}\big\|u\big\|_{S}.
\end{align*}
Then 
\begin{align*}
&\big\|P_0\big[u_{[k_1-10, -10]}\times (-\triangle)^{\frac12}(u_{k_1}\times (-\triangle)^{\frac12}u_{[-2,2]}) - u_{[k_1-10, -10]}\times (u_{k_1}\times (-\triangle u_{[-2,2]}))\big]\big\|_{L_t^1 \dot{H}_x^1}\\
&\lesssim \big\|u_{[k_1-10, -10]}\big\|_{L_t^2 L_x^\infty}\cdot 2^{\frac{k_1}{2}}\big\|u_{k_1}\big\|_{S}\big\|u_{[-2,2]}\big\|_{S}\lesssim \sum_{k_2\in [k_1-10, -10]}2^{\frac{k_1 - k_2}{2}}\big\|u_{k_2}\big\|_{S}\big\|u_{k_1}\big\|_{S}\big\|u\big\|_{S}\\
&\lesssim \big\|u_{k_1}\big\|_{S}\big\|u\big\|_{S}\big\|u_{[-2,2]}\big\|_{S},
\end{align*}
which is as desired. 
We omit the contributions when the first factor is replaced by $u_{>-10}$, which is handled similarly. 
\\
Thus consider now the more interesting term 
\begin{equation}\label{eq:harder}\begin{split}
&P_0\big[u_{<k_1-10}\times (-\triangle)^{\frac12}(u_{k_1}\times (-\triangle)^{\frac12}u) - u_{<k_1-10}\times (u_{k_1}\times (-\triangle u))\big]\\
&=P_0\big[u_{<k_1-10}\times (-\triangle)^{\frac12}(u_{k_1}\times (-\triangle)^{\frac12}u_0) - u_{<k_1-10}\times (u_{k_1}\times (-\triangle u_0))\big]\\
\end{split}\end{equation}
In order to handle it, we shall invoke the relation 
\begin{equation}\label{eq:crosspr}
a\times (b\times c) = b(a\cdot c) - c(a\cdot b), 
\end{equation}
but we have to take into account the non-local operator $(-\triangle)^{\frac12}$. Using the proof of Lemma \ref{lem:KS}, we can write 
\begin{align*}
&(-\triangle)^{\frac12}(u_{k_1}\times (-\triangle)^{\frac12}u_0) - (u_{k_1}\times (-\triangle u_0))\\
& = \sum_{m\in \Z^4, p\in \Z^4}a_{mp}u_{k_1}(x+2^{-k_1}m)\times u_0(x+p),
\end{align*}
where we have the bound $\big|a_{mp}\big|\lesssim_N 2^{k_1}\big[|m| + |p|\big]^{-N}$. Thus using \eqref{eq:crosspr} we can write \eqref{eq:harder} in the form 
\begin{align*}
&\sum_{m\in \Z^4, p\in \Z^4}a_{mp}u_{k_1}(x+2^{-k_1}m)\big(u_{<k_1-10}(x)\cdot u_0(x+p)\big)\\
& - \sum_{m\in \Z^4, p\in \Z^4}a_{mp}u_0(x+p)\big(u_{<k_1-10}(x)\cdot u_{k_1}(x+2^{-k_1}m)\big).
\end{align*}
Then the idea is to reduce to \eqref{eq:key}. For this write 
\begin{align*}
u_{<k_1-10}(x) = u_{<k_1-10}(x+p) - \int_0^1 p\cdot\nabla u_{<k_1-10}(x+ps)\,ds.
\end{align*}
Note that the integral term is bounded by 
\begin{align*}
\big\|\int_0^1 p\cdot\nabla u_{<k_1-10}(x+ps)\,ds\big\|_{L_t^2 L_x^\infty}\lesssim |p|\cdot \big\|\nabla  u_{<k_1-10}\big\|_{L_t^2 L_x^\infty},
\end{align*}
and so replacing the factor $u_{<k_1-10}$ by the integral term in the above sum, we get the bound 
\begin{align*}
&\big\|\sum_{m\in \Z^4, p\in \Z^4}a_{mp}u_{k_1}(x+2^{-k_1}m)\big((\int_0^1\ldots ds)\cdot u_0(x+p)\big)\big\|_{L_t^1 \dot{H}_x^1}\\
&\lesssim \sum_{m\in \Z^4, p\in \Z^4}a_{mp}\big\|u_{k_1}(x+2^{-k_1}m)\big\|_{L_t^2 L_x^\infty}\big\|\int_0^1\ldots ds\big\|_{L_t^2 L_x^\infty}\big\|u_0\big\|_{L_t^\infty L_x^2}\\
&\lesssim \sum_{m\in \Z^4, p\in \Z^4}|p|a_{mp}\|u_{k_1}\|_{S}\|u\|_S\|u_0\|_{S},
\end{align*}
where the sum over $m, p$ converges due to the rapid decay of $a_{mp}$. 
\\
It follows that we have reduced to estimating 
\begin{align*}
\sum_{m\in \Z^4, p\in \Z^4}a_{mp}u_{k_1}(x+2^{-k_1}m)\big(u_{<k_1-10}(x+p)\cdot u_0(x+p)\big),
\end{align*}
for which we can use \eqref{eq:key}. More precisely, write 
\begin{align*}
&u_{k_1}(x+2^{-k_1}m)\big(u_{<k_1-10}(x+p)\cdot u_0(x+p)\big)\\& = u_{k_1}(x+2^{-k_1}m)\big(u_{<-10}(x+p)\cdot u_0(x+p)\big)\\&
- u_{k_1}(x+2^{-k_1}m)\big(u_{[k_1-10, -10]}(x+p)\cdot u_0(x+p)\big)
\end{align*}
To bound the contribution from the second term on the right, we use 
\begin{align*}
&\big\|P_0\big[u_{k_1}(x+2^{-k_1}m)\big(u_{[k_1-10, -10]}(x+p)\cdot u_0(x+p)\big)\big]\big\|_{L_t^1 \dot{H}_x^1}\\
&\lesssim\big\|u_{k_1}(x+2^{-k_1}m)\big\|_{L_t^2 L_x^\infty}\big\|u_{[k_1-10, -10]}\big\|_{L_t^2 L_x^\infty}\big\|u_0\big\|_{L_t^\infty L_x^2}\\
&\lesssim \sum_{k_2\in [k_1-10, -10]}\|u_{k_1}\|_{S}\|u_{k_2}\|_S 2^{-\frac{k_1+k_2}{2}}\|u_0\|_{S}. 
\end{align*}
The coefficient $a_{mp}$ furnishes an extra $2^{k_1}$, and so inserting the previous bound we obtain
\begin{align*}
&\big\|\sum_{m\in \Z^4, p\in \Z^4}a_{mp}P_0[u_{k_1}(x+2^{-k_1}m)\big(u_{[k_1-10, -10]}(x+p)\cdot u_0(x+p)\big)]\big\|_{L_t^1\dot{H}_x^1}\lesssim \big\|u_{k_1}\big\|_{S}\big\|u\big\|_{S}\big\|u_{0}\big\|_{S}. 
\end{align*}
 It remains to deal with the contribution of 
\[
P_0[u_{k_1}(x+2^{-k_1}m)\big(u_{<-10}(x+p)\cdot u_0(x+p)\big)]. 
\]
Here we use \eqref{eq:key}, which gives 
\begin{align*}
\big\|P_0[u_{<-10}(x+p)\cdot u_0(x+p)]\big\|_{L_{t,x}^2}&\sim \big\|P_0[u_{\geq -10}(x+p)\cdot u_{\geq -10}(x+p)]\big\|_{L_{t,x}^2}\\
&\lesssim \big\|u\big\|_{S}^2, 
\end{align*}
and so we have 
\begin{align*}
&\big\|\sum_{m\in \Z^4, p\in \Z^4}a_{mp}P_0[u_{k_1}(x+2^{-k_1}m)\big(u_{<-10}(x+p)\cdot u_0(x+p)\big)]\big\|_{L_t^1\dot{H}_x^1}\\
&\lesssim \sum_{m\in \Z^4, p\in \Z^4}\big|a_{mp}\big| \big\|u_{k_1}\big\|_{L_t^2 L_x^\infty}\big\|u\big\|_{S}^2\lesssim 2^{\frac{k_1}{2}}\big\|u_{k_1}\big\|_{S}\big\|u\big\|_{S}^2. 
\end{align*}
It remains to deal with the expression 
\[
\sum_{m\in \Z^4, p\in \Z^4}a_{mp}u_0(x+p)\big(u_{<k_1-10}(x)\cdot u_{k_1}(x+2^{-k_1}m)\big),
\]
which is done in essentially identical fashion. 
\\

{\it{(ii): $k_1\in [-10, 10]$.}} Here we split things into 
\begin{align*}
&P_0\big[u\times (-\triangle)^{\frac12}(u_{k_1}\times (-\triangle)^{\frac12}u) - u\times (u_{k_1}\times (-\triangle u))\big]\\
& = P_0\big[u\times (-\triangle)^{\frac12}(u_{k_1}\times (-\triangle)^{\frac12}u_{\geq 20}) - u\times (u_{k_1}\times (-\triangle u_{\geq 20}))\big]\\
& + P_0\big[u\times (-\triangle)^{\frac12}(u_{k_1}\times (-\triangle)^{\frac12}u_{< 20}) - u\times (u_{k_1}\times (-\triangle u_{<20}))\big]
\end{align*}
Then we can write the first term as 
\begin{align*}
&P_0\big[u\times (-\triangle)^{\frac12}(u_{k_1}\times (-\triangle)^{\frac12}u_{\geq 20}) - u\times (u_{k_1}\times (-\triangle u_{\geq 20}))\big]\\
& = P_0\big[u_{>15}\times \left((-\triangle)^{\frac12}(u_{k_1}\times (-\triangle)^{\frac12}u_{\geq 20}) -  (u_{k_1}\times (-\triangle u_{\geq 20}))\right)\big],
\end{align*}
which is straightforward to estimate by placing $u_{>15}$  into $L_t^2 L_x^\infty$ and the other factor into $L_t^2 L_x^2$:
\begin{align}\label{eq:multiplier2}
\big\|(-\triangle)^{\frac12}(u_{k_1}\times (-\triangle)^{\frac12}u_{\ge20}) \big\|_{L_t^2 L_x^2}\lesssim \sum_{\ell\ge20} 2^{k_1+\ell} \norm{u_{k_1}}_{L_t^2 L_x^\infty} \norm{u_{\ell}}_{L_t^\infty L_x^2} \lesssim \norm{u}_S^2 
 \end{align}
and similarly for $u_{k_1} \times (-\triangle u_{\geq 20})$.

Then consider the more delicate term 
\[
P_0\big[u\times (-\triangle)^{\frac12}(u_{k_1}\times (-\triangle)^{\frac12}u_{< 20}) - u\times (u_{k_1}\times (-\triangle u_{<20}))\big].
\]
We treat 
\[
P_0\big[u_{\ge-10}\times (-\triangle)^{\frac12}(u_{k_1}\times (-\triangle)^{\frac12}u_{< 20}) - u_{\ge-10}\times (u_{k_1}\times (-\triangle u_{<20}))\big].
\]
similar to Equation \eqref{eq:multiplier2} above (but placing $u_{k_1}$ into $L_t^\infty L_x^2$ and $u_{<20}$ into $L_t^2 L_x^\infty$). Thus we are left with the term 
\[
P_0\big[u_{<-10}\times (-\triangle)^{\frac12}P_0(u_{k_1}\times (-\triangle)^{\frac12}u_{< 20}) - u_{<-10}\times P_0(u_{k_1}\times (-\triangle u_{<20}))\big].
\]

Similar to case~(i), write
\begin{align*}
&P_0\big((-\triangle)^{\frac12}(u_{k_1}\times (-\triangle)^{\frac12}u_{< 20}) - (u_{k_1}\times (-\triangle u_{< 20}))\big)\\
& = \sum_{m, p\in \Z^4}a_{mp}u_{k_1}(x+2^{-k_1}m)\times ((-\triangle)^{\frac12}u_{< 20})(x+2^{-20} p),
\end{align*}
where $|a_{mp}|\lesssim (|m| + |p|)^{-N}$. Thus using Equation \eqref{eq:crosspr} we can rewrite
\begin{align*}
u_{<-10}\times P_0\big((-\triangle)^{\frac12}&(u_{k_1}\times (-\triangle)^{\frac12}u_{< 20}) - u_{k_1}\times (-\triangle u_{<20})\big)
\\= \sum_{m, p\in \Z^4} a_{mp} \Big(&u_{k_1}(x+ 2^{-10} m)(u_{<-10}(x)\cdot (-\triangle)^{\frac12}u_{< 20}(x+2^{-20} p)) - \\&
((-\triangle)^{\frac12} u_{< 20})(x+2^{-20} p)(u_{k_1}(x+ 2^{-10} m) \cdot u_{<-10}(x))\Big).
\end{align*}
Let us estimate the first of the two terms in brackets. As before write $u_{<-10}(x)=u_{<-10}(x+2^{-20}p)- 2^{-20} p \cdot \int_0^1 \nabla u_{<-10}(x+2^{-20}ps) ds$. Then 
\begin{align*}
\norm{&\sum_{m\, p\in \Z^4} a_{mp} u_{k_1}(x+ 2^{-10} m)(u_{<-10}(x)\cdot ((-\triangle)^{\frac12}u_{< 20})(x+2^{-20} p)) }_{L_t^1 L_x^2} \lesssim \\&
\sum_{m, p\in \Z^4} |a_{mp}| \norm{u_{k_1}}_{L_t^{\infty} L_x^2} \left(\left\|p \cdot \int_0^1 \nabla u_{<-10}(x+2^{-20}ps) ds\right\|_{L_t^2 L_x^{\infty}}  \norm{(-\triangle)^{\frac12}u_{< 20}}_{L_t^2 L_x^{\infty}} + \norm{u_{<-10} \cdot (-\triangle)^{\frac12}u_{< 20}}_{L_t^1 L_x^{\infty}}\right).
\end{align*}
The integral term is bounded by $|p| \norm{\nabla u_{< -10}}_{L_t^2 L_x^{\infty}} \lesssim |p| \norm{u}_S$, the term $\norm{u_{<-10} \cdot (-\triangle)^{\frac12}u_{< 20}}_{L_t^1 L_x^{\infty}}$ is estimated using Equation \eqref{rescaled}. The sum over $m,p$ converges due to the decay of $|a_{mp}|$ and so the whole expression is bounded by $\norm{u}_S^3$. 

It remains to estimate 
\begin{align*}
\norm{&\sum_{m, p\in \Z^4} a_{mp} (-\triangle)^{\frac12} u_{< 20}(x+2^{-20} p)(u_{k_1}(x+ 2^{-10} m) \cdot u_{<-10}(x))}_{L_t^1 L_x^2}. 
\end{align*}
For this we write as before $u_{<-10}(x)= u_{<-10}(x+ 2^{-10} m)  - 2^{-10} m \cdot \int_0^1 \nabla u_{<-10}(x+ 2^{-10} ms ) ds$. The part involving the integral is easy to estimate. For the other part we put $(-\triangle)^{\frac12}u_{< 20}$ into $L_t^2 L_x^{\infty}$ and estimate 
\begin{align*}
  \norm{u_{[-10,10]}\cdot u_{<-10}}_{L_t^2 L_x^{2}}&\le \norm{u_{[-10,10]}\cdot u_{<-15}}_{L_t^2 L_x^{2}}+ \norm{u_{[-10,10]}\cdot u_{[-15,-10]}}_{L_t^2 L_x^{2}} \\& =
   \norm{P_{[-12,12]}(u_{[-10,10]}\cdot u_{<-15})}_{L_t^2 L_x^{2}}+ \norm{u_{[-10,10]}\cdot u_{[-15,-10]}}_{L_t^2 L_x^{2}} \\&
  \lesssim \norm{P_{[-12,12]}(u_{\ge -15}\cdot u_{<-15})}_{L_t^2 L_x^{2}} + \norm{u}_S^2  \\&
  \sim \norm{P_{[-12,12]}(u_{\ge -15}\cdot u_{\ge -15})}_{L_t^2 L_x^{2}} + \norm{u}_S^2 \lesssim \norm{u}_S^2 
\end{align*}
where in the last line we have used Equation \eqref{eq:key_general} to turn the low-frequency factor into a high frequency one.

{\it{(ii): $k_1 > 10$.}} We decompose
\begin{align}
 \begin{split}\label{eq:lastcase}
  P_0\big[u\times & \big( (-\triangle)^{\frac12}(u_{k_1}\times (-\triangle)^{\frac12}u) -  (u_{k_1}\times (-\triangle u))\big)\big]\\
= & P_0\big[u \times \big((-\triangle)^{\frac12}(u_{k_1}\times (-\triangle)^{\frac12}u_{<k_1-5}) - (u_{k_1}\times (-\triangle u_{<k_1-5}))\big)\big]\\
 + & P_0\big[u \times \big((-\triangle)^{\frac12}(u_{k_1}\times (-\triangle)^{\frac12}u_{[k_1-5, k_1+5]}) - (u_{k_1}\times (-\triangle u_{[k_1-5, k_1+5]}))\big)\big]\\
 + &P_0\big[u \times \big((-\triangle)^{\frac12}(u_{k_1}\times (-\triangle)^{\frac12}u_{>k_1+5}) - (u_{k_1}\times (-\triangle u_{>k_1+5}))\big)\big].
 \end{split}
\end{align}


The first and the last line are easy to handle since in both cases the first factor $u$ can be restricted to frequencies $\gtrsim k_1$ and placed into $L_t^2 L_x^{\infty}$ and the other factor can be placed into $L_t^2 L_x^{2}$ e.g. 
\begin{align*}
 &\norm{(-\triangle)^{\frac12}(u_{k_1}\times (-\triangle)^{\frac12}u_{<k_1-5})}_{L_t^2 L_x^2} \lesssim \sum_{k_2 <k_1-5} 2^{k_1} \norm{u_{k_1}}_{L_t^\infty L_x^2}\norm{(-\triangle)^{\frac12}u_{k_2}}_{L_t^2 L_x^\infty}\lesssim 2^{-\frac{k_1}{2}} \norm{u_{k_1}}_S \norm{u}_S; \\ 
 &\norm{u_{k_1}\times (-\triangle u_{>k_1+5})}_{L_t^2 L_x^2} \lesssim \sum_{k_2 >k_1+5} \norm{u_{k_1}}_{L_t^2 L_x^\infty}\norm{\triangle u_{k_2}}_{L_t^\infty L_x^2}\lesssim \norm{u_{k_1}}_S\norm{u}_S; \quad \text{etc.}
\end{align*}
Hence we have bounded the first and the last line in Equation \eqref{eq:lastcase} by $2^{-\frac{k_1}{2}}  \norm{u_{k_1}}_S \norm{u}_S^2$. 

Finally, for the term in the middle we first bound
\begin{align*}
 \norm{P_0& \big[u_{\ge -5} \times \big((-\triangle)^{\frac12}(u_{k_1}\times (-\triangle)^{\frac12}u_{[k_1-5, k_1+5]}) - (u_{k_1}\times (-\triangle u_{[k_1-5, k_1+5]}))\big)\big]}_{L_t^1 L_x^2} \\& 
 \lesssim \norm{u_{\ge -5}}_{L_t^2 L_x^{\infty}} 2^{2k_1} \norm{u_{k_1}}_{L_t^2 L_x^{\infty}}  \norm{u_{[k_1-5, k_1+5]}}_{L_t^\infty L_x^{2}}  \lesssim 2^{-\frac{k_1}{2}} \norm{u_{k_1}}_S \norm{u}_S^2.
\end{align*}
It remains to estimate 
\begin{equation}\label{eq:lastone}
 \norm{P_0 \big[u_{<-5} \times \big((-\triangle)^{\frac12}(u_{k_1}\times (-\triangle)^{\frac12}u_{[k_1-5, k_1+5]}) - (u_{k_1}\times (-\triangle u_{[k_1-5, k_1+5]}))\big)\big]}_{L_t^1 L_x^2}.
\end{equation}
For the second term we can directly apply Equation \eqref{eq:crosspr} to obtain
\begin{equation}\label{eq:secondterm}
 \norm{P_0 \big[u_{<-5} \times (u_{k_1}\times (-\triangle u_{[k_1-5, k_1+5]}))\big]}_{L_t^1 L_x^2} = \norm{P_0\big[u_{k_1}(u_{<-5}\cdot (-\triangle)u_{[k_1-5, k_1+5]})-(-\triangle)u_{[k_1-5, k_1+5]} (u_{<-5}\cdot u_{k_1} )\big]}_{L_t^1 L_x^2}
\end{equation}
Here the second term can be estimated easily after applying Equation \eqref{eq:key_general} to turn the low frequency $u$ into a high-frequency one: 
$$\norm{u_{<-5}\cdot u_{k_1}}_{L_t^1 L_x^\infty} \le \sum_{k\in [k_1-1,k_1+1]} \norm{P_k(u_{<-5} \cdot u)}_{L_t^1 L_x^\infty}\sim \sum_{k\in [k_1-1,k_1+1]} \norm{P_k(u_{\ge -5} \cdot u_{\ge -5})}_{L_t^1 L_x^\infty}\lesssim \norm{u}_S^2$$
For the first term we use Equation \eqref{rescaled2}:
\begin{align*}
 \norm{u_{<-5}\cdot &(-\triangle)u_{[k_1-5, k_1+5]}}_{L_t^2 L_x^2} \le \sum_{k\in [k_1-6,k_1+6]} \norm{P_k(u_{<-5} \cdot (-\triangle)u)}_{L_t^2 L_x^2}\\ &\lesssim 
 \sum_{k\in [k_1-6,k_1+6]} \Big( \norm{P_k(-\triangle)(u_{<-5} \cdot u)}_{L_t^2 L_x^2}+  \sum_{\substack{\ell <-5\\ k_2\in [k-2,k+2]}} 2^{\ell+k} \big\| u_{\ell} \big\|_{L_t^2 L_x^\infty}\big\|u_{k_2}\big\|_{L_t^\infty L_x^2}\Big)
\end{align*}
which after again applying Equation \eqref{eq:key} we bound by $\norm{u}_S^2$.

Finally, for the first term in Equation \eqref{eq:lastone} we  expand the Fourier multiplier $m(\xi,\eta):= \chi_0(\xi + \eta)|\eta| $ as a Fourier series on the support of $\chi_{< -5}(\xi) \chi_{[-1,1]}(\eta)$ to obtain 
\begin{align*}
 P_0 \big[u_{< -5} &\times (-\triangle)^{\frac12}(u_{k_1}\times (-\triangle)^{\frac12}u_{[k_1-5, k_1+5]})\big](x)= 
\\&\sum_{m,p \in \Z^4} a_{mp} u_{< -5}(x+ 2^{-5}m) \times (u_{k_1}(x+2^{-1}p)\times (-\triangle)^{\frac12}u_{[k_1-5, k_1+5]}(x+2^{-1}p))
\end{align*}
 where $|a_{mp}| \lesssim_N (|m|+ |p|)^{-N}$. As in the previous cases we then write 
 $$u_{< -5}(x+ 2^{-5}m) = u_{< -5}(x+ 2^{-1}p)+(2^{-5}m-2^{-1}p) \int_0^1 \nabla u_{< -5} (x+(2^{-5}m-2^{-1}p)s) ds.$$
The part involving the integral is easy to estimate.
Thus we are left with estimating 
$$\norm{P_0 \big[u_{<-5}\times (u_{k_1} \times (-\triangle)^{\frac12}u_{[k_1-5, k_1+5]})\big]}_{L_t^1 L_x^2}.$$
This is done in the same way as above (Equation \eqref{eq:secondterm}) by writing the cross product in terms of inner products and then turning the low frequency $u$ into a high frequency one.  

\end{proof}

The preceding lemmas~\ref{lem:complicatedterm},~\ref{lem:multilin2} after re-scaling immediately yield \eqref{eq:multilin2}, \eqref{eq:multilin3},  and the difference estimate \eqref{eq:multilin5} together with the corresponding difference estimate corresponding to \eqref{eq:multilin3} follows by using the same proofs, involving the use of \eqref{eq:key}, and passing to the differences after this. Furthermore, one uses the important fact that if $\lim_{N\rightarrow +\infty}P_{<-N}\tilde{u}^{(j)} = p\in S^2$, then 
\[
\|\tilde{u}^{(j)} - \tilde{u}^{(i)}\|_{L^\infty}\lesssim \|\tilde{u}^{(j)} - \tilde{u}^{(i)}\|_{S}, 
\]
as follows from Bernstein's inequality. This proves Proposition~\ref{prop:multilinear}.

\section{The iteration}

Here we use the results of the preceding section to construct a global solution for \eqref{eq:halfwmwaveform}. Specifically, we have 
\begin{prop}\label{prop:waveform} Let $u[0] = (u(0, \cdot), u_t(0, \cdot) = (u_0, u_1): \R^{4}\longrightarrow S^2\times TS^2$ where $u_1 = u_0\times (-\triangle)^{\frac12} u_0$ and $u_0$ is $C^\infty$-smooth, constant outside of a compact subset of $\R^4$, and  
\[
\|u_0\|_{\dot{B}^{2,1}_{2}}<\epsilon
\]
for sufficiently small $\epsilon>0$. Then the problem \eqref{eq:halfwmwaveform} admits a global $C^\infty$ solution with initial data $u[0]$. Furthermore, the solution satisfies 
\[
\|u\|_{S}\lesssim \epsilon,
\]
and $u$ scatters in the sense that for suitable $(f_{\pm},g_{\pm})\in \dot{B}^{2,1}_{2}\times \dot{B}^{1,1}_{2}$, we have 
\[
u(t, \cdot) - p = S(t)\big(f_{\pm}, g_{\pm}) + o_{\dot{B}^{2,1}_{2}\times \dot{B}^{1,1}_{2}}(1)
\]
as $t\rightarrow\pm \infty$. Finally, the solution also satisfies the bounds 
\[
\|\nabla_{t,x}^ku\|_{S}\lesssim_k 1
\]
for any $k\geq 1$. 
\end{prop}
\begin{rem} One can easily generalise the condition on the second component of the data $u_1$ in this proposition. However, for the application in the proof of Theorem~\ref{thm:Main}, we need precisely this type of data. 

\end{rem}
\begin{proof} To begin with, observe that we have 
\[
\|u_0\times (-\triangle)^{\frac12}u_0\|_{\dot{B}^{1,1}_{2}}\lesssim \|u_0\|_{\dot{B}^{2,1}_{2}}.
\]
We follow the same iterative scheme as in [KS], which we recall here for completeness' sake. The solution $u(t, x)$ shall be obtained as the limit of a sequence of iterates $u^{(j)}$, starting with the constant function $p:=\lim_{|x|\rightarrow\infty}u_0(x)$. Next, $u^{(1)}(t, x)$ solves the wave maps equation with data $u[0]$, i. e. it solves 
\[
(\partial_t^2 - \triangle) u^{(1)} =  u^{(1)}(\nabla u^{(1)}\cdot\nabla u^{(1)} - \partial_t u^{(1)}\cdot\partial_t u^{(1)}),\,u^{(1)}[0] = u[0].
\]
This is obtained in the form $u^{(1)} = p + \sum_{k\in Z}u^{(1)}_k$ by invoking Proposition \ref{prop:linestimates} and the estimates \eqref{eq:multilin1}, \eqref{eq:multilin4}. Further, for $j\geq 2$, we let $u^{(j)}$ solve the system 
 \begin{equation*}\label{eq:iterate}\begin{split}
&(\partial_t^2 - \triangle) u^{(j)}\\&= u^{(j)}(\nabla u^{(j)}\cdot\nabla u^{(j)} - \partial_t u^{(j)}\cdot\partial_t u^{(j)})\\
&+\Pi_{u_{\perp}^{(j)}}\big((-\triangle)^{\frac{1}{2}}u^{(j-1)}\big)(u^{(j-1)}\cdot (-\triangle)^{\frac{1}{2}}u^{(j-1)})\\
&+\Pi_{u_{\perp}^{(j)}}\big[u^{(j-1)} \times (-\triangle)^{\frac{1}{2}}(u^{(j-1)} \times (-\triangle)^{\frac{1}{2}}u^{(j-1)}) - u^{(j-1)} \times(u^{(j-1)} \times (-\triangle)u^{(j-1)})\big]
\end{split}\end{equation*} 
As $u^{(j)}$ is defined in terms of a nonlinear problem, we have to use a sub-iteration to find it, by means of further iterates $u^{(j,i)}$, $i\geq 0$, with $u^{(j,0)}$ solving the free wave equation with data $u[0]$, and $u^{(j,i)}$, $i\geq 1$, defined in terms of the previous iterate via 
 \begin{equation}\label{eq:iterate1}\begin{split}
&(\partial_t^2 - \triangle) u^{(j,i)}\\&= u^{(j,i-1)}(\nabla u^{(j,i-1)}\cdot\nabla u^{(j,i-1)} - \partial_t u^{(j,i-1)}\cdot\partial_t u^{(j,i-1)})\\
&+\Pi_{u_{\perp}^{(j,i-1)}}\big((-\triangle)^{\frac{1}{2}}u^{(j-1)}\big)(u^{(j-1)}\cdot (-\triangle)^{\frac{1}{2}}u^{(j-1)})\\
&+\Pi_{u_{\perp}^{(j,i-1)}}\big[u^{(j-1)} \times (-\triangle)^{\frac{1}{2}}(u^{(j-1)} \times (-\triangle)^{\frac{1}{2}}u^{(j-1)}) - u^{(j-1)} \times(u^{(j-1)} \times (-\triangle)u^{(j-1)})\big].
\end{split}\end{equation} 
The convergence of \eqref{eq:iterate1} with respect to $\|\cdot\|_{S}$ is a consequence of \eqref{eq:multilin1} - \eqref{eq:multilin4}, provided one checks that all $u^{(j)}$ map into the sphere. This is clear for $u^{(j)}, j = 0,1$, and follows from 
\[
\Box(u^{(j)}\cdot u^{(j)} - 1) = 2(u^{(j)}\cdot u^{(j)} - 1)(\nabla u^{(j)}\cdot\nabla u^{(j)} - \partial_t u^{(j)}\cdot\partial_t u^{(j)}),\,(u^{(j)}\cdot u^{(j)} - 1)[0] = 0
\]
for all $j\geq 2$. Observe that the $u^{(j)}$ are in fact $C^\infty$ by differentiating the equation and again using \eqref{eq:multilin1} - \eqref{eq:multilin4}. Convergence of the iterates $u^{(j)}$ then also follows from \eqref{eq:multilin1} - \eqref{eq:multilin4} by passing to the difference equations. Finally, the scattering assertion follows in standard fashion from the fact that the right hand source terms in \eqref{eq:halfwmwaveform} can all be placed in $N$, see e. g. \cite{Tat3}. 

\end{proof}

\section{Proof of Theorem~\ref{thm:Main}}

Given smooth $u_0\rightarrow S^2$ as in the statement of Theorem~\ref{thm:Main}, we apply Proposition~\ref{prop:waveform} to obtain a global $C^\infty$ solution for \eqref{eq:halfwmwaveform} with initial data $u[0] = \big(u_0, u_0\times (-\triangle)^{\frac12}u_0)\big)$. It remains to show that this $u$ actually solves \eqref{eq:halfwm}. We proceed as in \cite{KS}, section 6, but need to modify the estimates there somewhat. Introduce 
\[
X: = u_t - u\times (-\triangle)^{\frac12}u,
\]
as well as the energy type functional 
\[
\tilde{E}(t): = \frac12\int_{\R^4}\big|(-\triangle)^{\frac{1}{4}}X(t,\cdot)\big|^2\,dx.
\]
Observe that directly integrating \eqref{eq:halfwm} over compact time intervals and using the compact support of $\nabla_x u_0$,  we have that $\|u_t\|_{\dot{H}^{\frac12}} + \|u\|_{\dot{H}^{\frac32}}<\infty$ at all times, which implies that $\tilde{E}(t)$ is well-defined, and smoothness of $u$ implies that it can be differentiated with respect to $t$. Proceeding as in \cite{KS}, section 6, we deduce the relation 
\begin{align*}
\partial_tX = -X\times  (-\triangle)^{\frac12}u  - u\times(-\triangle)^{\frac12}X - u\big(X\cdot(u\times  (-\triangle)^{\frac12}u + u_t)\big),
\end{align*}
and so we infer 
\begin{equation}\label{eq:technical10}\begin{split}
\frac{d}{dt}\tilde{E}(t) =& -\int_{\R^4}(-\triangle)^{\frac{1}{4}}\big(X\times  (-\triangle)^{\frac12}u +  u\times(-\triangle)^{\frac12}X\big)\cdot (-\triangle)^{\frac{1}{4}}X\,dx\\
& - \int_{\R^4}(-\triangle)^{\frac{1}{4}}\big(u\big(X\cdot(u\times  (-\triangle)^{\frac12}u + u_t)\big)\big)\cdot(-\triangle)^{\frac{1}{4}}X\,dx.\\
\end{split}\end{equation}
Then taking advantage of Lemma~\ref{lem:KS} we find 
\begin{align*}
&\big\|(-\triangle)^{\frac{1}{4}}\big(X\times  (-\triangle)^{\frac12}u +  u\times(-\triangle)^{\frac12}X\big) - u\times(-\triangle)^{\frac{3}{4}}X \big\|_{L_x^2}\\&\lesssim \big\|(-\triangle)^{\frac{1}{4}}X\big\|_{L_x^2} \big\| (-\triangle)^{\frac12}u\big\|_{L_x^\infty} + \big\|X\big\|_{L_x^{\frac{8}{3}}} \big\|(-\triangle)^{\frac{3}{4}}u\big\|_{L_x^{8}}\\
&\lesssim_u \big\|(-\triangle)^{\frac{1}{4}}X\big\|_{L_x^2},
\end{align*}
where Sobolev's inequality and the higher regularity of $u$ are being used. 
\\
Another application of Lemma~\ref{lem:KS} gives 
\[
u\times(-\triangle)^{\frac{3}{4}}X = (-\triangle)^{\frac14}\big(u\times(-\triangle)^{\frac{2}{4}}X\big) + O_{L_x^2}\big(\|\nabla_x u\|_{L_x^\infty}\|(-\triangle)^{\frac{1}{4}}X\|_{L_x^2}\big),
\]
and since  
\[
\int_{\R^4}(-\triangle)^{\frac14}\big(u\times(-\triangle)^{\frac{2}{4}}X\big)\cdot (-\triangle)^{\frac{1}{4}}X\,dx = 0,
\]
we can estimate the first term on the right of \eqref{eq:technical10} by 
\begin{align*}
&\big|\int_{\R^4}(-\triangle)^{\frac{1}{4}}\big(X\times  (-\triangle)^{\frac12}u +  u\times(-\triangle)^{\frac12}X\big)\cdot (-\triangle)^{\frac{1}{4}}X\,dx\big|\\
&\leq\big|\int_{\R^4}(-\triangle)^{\frac{1}{4}}\big(X\times  (-\triangle)^{\frac12}u +  u\times(-\triangle)^{\frac12}X- u\times(-\triangle)^{\frac{3}{4}}X\big)\cdot (-\triangle)^{\frac{1}{4}}X\,dx\big|\\
& + O\big(\|\nabla_x u\|_{L_x^\infty}\|(-\triangle)^{\frac{1}{4}}X\|_{L_x^2}^2\big)\\
&\lesssim_{u} \|(-\triangle)^{\frac{1}{4}}X\|_{L_x^2}^2.
\end{align*}
The second term on the right hand side of \eqref{eq:technical10} is estimated similarly, using 
\[
\|X\cdot(u\times  (-\triangle)^{\frac12}u + u_t)\|_{L_x^2}\lesssim_{u}  \|(-\triangle)^{\frac{1}{4}}X\|_{L_x^2}. 
\]
We conclude that 
\[
\frac{d}{dt}\tilde{E}(t) \lesssim_u \tilde{E}(t), 
\]
and in light of $\tilde{E}(0) = 0$, we infer $\tilde{E}(t) = 0$ for all $t$, which implies that $u$ indeed solves \eqref{eq:halfwm}.

\medskip
\medskip

\centerline{\scshape Anna Kiesenhofer}
\medskip
{\footnotesize
 \centerline{B\^{a}timent des Math\'ematiques, EPFL}
\centerline{Station 8, 
CH-1015 Lausanne, 
  Switzerland}
  \centerline{\email{anna.kiesenhofer@epfl.ch}}
} 

\medskip

\centerline{\scshape Joachim Krieger }
\medskip
{\footnotesize
 \centerline{B\^{a}timent des Math\'ematiques, EPFL}
\centerline{Station 8, 
CH-1015 Lausanne, 
  Switzerland}
  \centerline{\email{joachim.krieger@epfl.ch}}
} 

\end{document}